\documentclass[11pt]{article}
\usepackage[centertags]{amsmath}
\usepackage{amsfonts}
\usepackage{amssymb}
\usepackage{makeidx}
\usepackage{newlfont}
\usepackage{amsthm} 
\usepackage{stmaryrd}
\usepackage[]{titletoc}  
\usepackage{mathrsfs}
\usepackage{stmaryrd}
 \usepackage{graphicx}  
 \usepackage{xypic}
\usepackage{hyperref}

\newlength{\defbaselineskip}
\newcommand{\setlinespacing}[1]%
           {\setlength{\baselineskip}{#1 \defbaselineskip}}

\theoremstyle{plain}
\newtheorem{thm}{Theorem}[section]
\newtheorem{cor}[thm]{Corollary}
\newtheorem{lem}[thm]{Lemma}
\newtheorem{prop}[thm]{Proposition}

\newtheorem{Qes}[thm]{Question}

\makeatletter\@addtoreset{equation}{section} \makeatother
\begin{document}

\title {Power dilation systems $\{f(z^k)\}_{k\in\mathbb{N}}$ in Dirichlet-type spaces}
\author{
Hui Dan \quad Kunyu Guo
 }
\date{}
 \maketitle \noindent\textbf{Abstract:}
 In this paper, we concentrate on power dilation systems $\{f(z^k)\}_{k\in\mathbb{N}}$ in Dirichlet-type spaces $\mathcal{D}_t\ (t\in\mathbb{R})$.
 When $t\neq0$, we  prove that $\{f(z^k)\}_{k\in\mathbb{N}}$ is orthogonal in $\mathcal{D}_t$ only if
 $f=cz^N$ for some constant $c$ and some positive integer $N$.
 We  also  give  complete characterizations of  unconditional bases and frames formed by power dilation systems for Drichlet-type spaces.



\section{Introduction}

In three papers, published in the 1940s, by Bourgin and Mendel \cite{BM} and by Bourgin \cite{Bour1,Bour2}, a fundamental class of  functions in the Lesbegue space $L^2(0,1)$ was studied. That is,
those functions $\psi\in L^2(0,1)$ satisfying the following
orthogonality relations:
$$\langle \psi(kx),\psi(lx)\rangle=\delta_{kl},\quad k,l=1,2,\cdots,$$ where $\psi$ is identified with its odd $2$-periodic extension on the real line $\mathbb{R}$.
They first found that there exists a function $\psi\in L^2(0,1)$, which is not a constant multiple  of $\sin N\pi x$ for some ${N\in\mathbb{N}}$, such that  $\{\psi(kx)\}_{k\in\mathbb{N}}$  is orthogonal  in   $L^2(0,1)$ \cite{BM}.
Although they obtained lots of results on this subject,
they did not completely characterize such a class of functions.
In these papers, they also studied the problem
that for which $\varphi\in L^2(0,1)$,
the periodic dilation system $\{\varphi(kx)\}_{k\in\mathbb{N}}$
is complete, namely  $\{\varphi(kx)\}_{k\in\mathbb{N}}$ spans a dense linear subspace of   $L^2(0,1)$. This problem was first considered by Wintner \cite{Win} and Beurling \cite{Beu}.

The classical treatment for the above problems is to associate $L^2(0,1)$-functions with Dirichlet series by using coefficients of their Fourier-sine expansions (see \cite{Win,Beu,Sei}). This was systematized in the frame of Hilbert spaces due to Hedenmalm,  Lindqvist and  Seip's work in \cite{HLS}. More precisely, they introduced the Hardy space of Dirichlet series  $$\mathcal{H}^2=\{f=\sum_{n=1}^\infty a_nn^{-s}:\|f\|^2=\sum_{n=1}^\infty|a_n|^2<\infty\},$$
and studied periodic dilation systems in $L^2(0,1)$ via a unitary transform \begin{eqnarray} \label{operator U}
                                                 U:\quad L^2(0,1)& \rightarrow &
  \mathcal{H}^2 , \notag\\
  \sum_{n=1}^\infty a_n\varphi_n& \mapsto & \sum_{n=1}^\infty a_nn^{-s}. \end{eqnarray}
Here $\varphi_n(t)=\sqrt{2}\sin n\pi t$, and $\{\varphi_n\}_{n\in\mathbb{N}}$ is a canonical orthonormal basis for $L^2(0,1)$.

Building on Beurling's ideas \cite{Beu},
they also used the Bohr transform  to establish a connection between the space $\mathcal{H}^2$ and the Hardy space $H^2(\mathbb{D}_2^\infty)$ over the infinite polydisk, which allows  applications of the Hardy space technique to
problems in $L^2(0,1)$. To be more specific, let $p_m\ (m\in\mathbb{N})$ denote the $m$-th prime,
and
for each positive integer $n$ with prime factorization $n=p_1^{\alpha_1}\cdots p_m^{\alpha_m}$, 
write $\zeta^{\alpha(n)}$ for the monimial $\zeta_1^{\alpha_1}\cdots\zeta_m^{\alpha_m}$.
The Hardy space $H^2(\mathbb{D}_2^\infty)$ is defined to be the Hilbert space consisting of formal power series
$F(\zeta)=\sum_{n=1}^\infty a_n \zeta^{\alpha(n)} $
 satisfying
$\|F\|^2= \sum_{n=1}^\infty |a_n|^2 <\infty.$
Setting
$$\mathbb{D}_2^\infty=\{\zeta=(\zeta_1,\zeta_2,\cdots)\in l^2:|\zeta_n|<1\  \mathrm{for}\ \mathrm{all}\ n\geq1\}$$
and applying the Cauchy-Schwarz inequality, one see that every series in $H^2(\mathbb{D}_2^\infty)$ converges absolutely in $\mathbb{D}_2^\infty$ and is
holomorphic in $\mathbb{D}_2^\infty$.
A brilliant observation by Bohr is that a Dirichlet series can be transformed into
 a power series  in infinitely many variables
via the``variable" substitution \cite{Bo}
 $$\zeta_1=p_1^{-s},\zeta_2=p_2^{-s},\cdots,$$
 which defines  the unitary transform \cite{HLS}
 \begin{eqnarray} \label{operator B}\mathbf{B}:\mathcal{H}^2& \rightarrow &
  H^2(\mathbb{D}_2^\infty) , \notag\\
  \sum_{n=1}^\infty a_nn^{-s}& \mapsto &\sum_{n=1}^\infty a_n\zeta^{\alpha(n)}.\end{eqnarray}

By the unitary transform $U$ defined in (\ref{operator U}), we can now
assign a special class of Dirichlet series in $\mathcal{H}^2$ --- $\mathcal{H}^2$-inner functions   to orthogonal dilation systems $\{\varphi(kx)\}_{k\in\mathbb{N}}$ in $L^2(0,1)$.
A Dirichlet series $D\in \mathcal{H}^2$ is said to be $\mathcal{H}^2$-inner if
$\{k^{-s}D\}_{k\in\mathbb{N}}$ is orthogonal in  $\mathcal{H}^2$.
This notion is introduced
 by  Olofsson in a more general setting \cite{O}. He indicated that every $\mathcal{H}^2$-inner function can be identified with
an isometric multiplier on $\mathcal{H}^2$, and further an inner function in the Hardy space $H^2(\mathbb{D}_2^\infty)$  via the Bohr transform $\mathbf{B}$ in (\ref{operator B}) (see Section 2 for details).
For complete periodic  dilation systems in $L^2(0,1)$, a complete characterization in terms of cyclic vectors for multiplier algebras of both functions spaces $\mathcal{H}^2$ and $H^2(\mathbb{D}_2^\infty)$  was given in \cite{HLS}: $\{\varphi(kx)\}_{k\in\mathbb{N}}$ is complete in $L^2(0,1)$
if and only if $U\varphi$ is cyclic in $\mathcal{H}^2$, if and only if $\mathbf{B}U\varphi$ is cyclic in $H^2(\mathbb{D}_2^\infty)$.
Besides the completeness problem, Riesz basis for $L^2(0,1)$ formed by dilation systems $\{\psi(kt)\}_{k\in\mathbb{N}}$ was also completely
characterized in \cite{HLS}. Riesz bases are known as the most tractable class of bases for Hilbert spaces.
They gave the following
sufficient and necessary condition for Riesz bases: $U\psi$
is  invertible  in the multiplier algebra  of $\mathcal{H}^2$, equivalently, $$0<c\leq|\mathbf{B}U\psi(\zeta)|\leq C,\quad\zeta\in\mathbb{D}_2^\infty$$ for some  constants $c,C$.

The completeness problem was also studied by Nikolski in the context of the Hardy space $H^2(\mathbb{D})$ over the unit disk.
To link $L^2(0,1)$ and $H^2(\mathbb{D})$, he use the unitary transform $V:H_0^2(\mathbb{D})\rightarrow L^2(0,1)$ defined by $Vz^n=\sqrt{2}\sin n\pi t$ $(n\in\mathbb{N})$, where $$H_0^2(\mathbb{D})=\{f\in H^2(\mathbb{D}):f(0)=0\}.$$
The Bohr transform for the Hardy space was also given in \cite{Ni}:
\begin{eqnarray*}\mathcal{B}:\quad H_0^2(\mathbb{D})& \rightarrow &
  H^2(\mathbb{D}_2^\infty),\\
  \sum_{n=1}^\infty a_nz^n& \mapsto & \sum_{n=1}^\infty a_n\zeta^{\alpha(n)},\end{eqnarray*}
  which satisfies $\mathbf{B}U=\mathcal{B}V^{-1}$ on $L^2(0,1)$.
A simple verification gives that $V^{-1}$ maps a dilation system
$\{\psi(kt)\}_{k\in\mathbb{N}}$ in $L^2(0,1)$ into a power dialtion system $\{(V^{-1}\psi)(z^k)\}_{k\in\mathbb{N}}$ in $H_0^2(\mathbb{D})$. Thus
all afore-mentioned  results are naturally translated into the analogous ones for power dilation systems $\{f(z^k)\}_{k\in\mathbb{N}}$ in
$H_0^2(\mathbb{D})$.
\begin{thm}[\cite{O,HLS,Ni}]\label{collection thm} Suppose that $f$ is a nonzero function in $H_0^2(\mathbb{D})$. Then
 \begin{itemize}
  \item [(1)] $\{f(z^k)\}_{k\in\mathbb{N}}$ is orthogonal in
$H_0^2(\mathbb{D})$ if and only if $\mathcal{B}f$ is a constant multiple of  inner functions in $H^2(\mathbb{D}_2^\infty)$;
  \item [(2)] $\{f(z^k)\}_{k\in\mathbb{N}}$ is complete in
$H_0^2(\mathbb{D})$ if and only if $\mathcal{B}f$ is cyclic in $H^2(\mathbb{D}_2^\infty)$;
  \item [(3)] $\{f(z^k)\}_{k\in\mathbb{N}}$ is a Riesz basis for
$H_0^2(\mathbb{D})$ if and only if $\mathcal{B}f$ is bounded from both above and below on $\mathbb{D}_2^\infty$.
\end{itemize}
\end{thm}

Note that for $f=cz^N\ (c\in\mathbb{C},N\in\mathbb{N})$, $\{f(z^k)\}_{k\in\mathbb{N}}$ is always orthogonal in
$H_0^2(\mathbb{D})$. Furthermore, one can use Theorem \ref{collection thm} (1) to construct various nontrivial examples for orthogonal power dilation systems in
$H_0^2(\mathbb{D})$. In fact, given $a\in\mathbb{D}$, we consider  the inner function
$$F(\zeta)=\frac{a-\zeta_1}{1-\bar{a}\zeta_1}
=a-(1-|a|^2)\sum_{n=1}^{\infty}\bar{a}^{n-1}\zeta_1^n,\quad  \zeta\in\mathbb{D}_2^\infty,$$
and put $$f=\mathcal{B}^{-1}F=az-(1-|a|^2)\sum_{n=1}^{\infty}\bar{a}^{n-1}z^{2^n}.$$
Then  $\{f(z^k)\}_{k\in\mathbb{N}}$ is orthogonal in
$H_0^2(\mathbb{D})$.

A natural question is whether similar results on power dilation systems hold  in the
Bergman space $L_{a,0}^2(\mathbb{D})=L_a^2(\mathbb{D})\ominus\mathbb{C}$, or general Dirichlet-type spaces $\mathcal{D}_t$ on the unit disk.
\begin{Qes}
For which function $f\in L_{a,0}^2(\mathbb{D})$ (or $\mathcal{D}_t$), is the system $\{f(z^k)\}_{k\in\mathbb{N}}$  orthogonal, complete, or a Riesz basis?
\end{Qes}
The Dirichlet-type space $\mathcal{D}_t\ (t\in\mathbb{R})$  is defined as
$$
\mathcal{D}_t=\{f=\sum_{n=1}^\infty a_nz^n\in \mathrm{Hol}(\mathbb{D}):\|f\|_t^2=\sum_{n=1}^\infty|a_n|^2(n+1)^t<\infty\},
$$ where $\mathrm{Hol}(\mathbb{D})$ denote the set of holomorphic functions on $\mathbb{D}$.
In particular, $\mathcal{D}_0$ is the Hardy space $H_0^2(\mathbb{D})$,  $\mathcal{D}_{-1}$ is the Bergman space $L_{a,0}^2(\mathbb{D})$, and  $\mathcal{D}_1$ is the (ordinary) Dirichelt space.
For each $t\in\mathbb{R}$, the only orthogonal basis for $\mathcal{D}_t$ formed by a power dilation system $\{f(z^k)\}_{k\in\mathbb{N}}$  is obtained from the identity map $f(z)=z$ up
  to a constant multiple. However, unlike the Hardy space case, orthogonal power dilation systems in other Dirichelt-type spaces, as our main result stated below, only have trivial forms.
\begin{thm}\label{t1}
Suppose $t\neq0$ and $f\in \mathcal{D}_t$. Then  $\{f(z^k)\}_{k\in\mathbb{N}}$ is orthogonal in $\mathcal{D}_t$ if and only if $f(z)=cz^N$ for some constant $c$ and some positive integer $N$.
\end{thm}

As we will see in Section 3,  the completeness
problems of power dilation systems on Dirichlet-type spaces are all equivalent.
 On the other hand, when $t\neq0$, a power dilation system in $\mathcal{D}_t$ never forms a Riesz basis. This is because a Riesz basis is norm-bounded from both above and below, while every nonzero system $\{f(z^k)\}_{k\in\mathbb{N}}$ in $\mathcal{D}_t$ is  not for $t\neq0$.
However, we found that each $\mathcal{D}_t$
has many unconditional bases formed by power dilation systems.
So it should make more sense to find conditions when $\{f(z^k)\}_{k\in\mathbb{N}}$ forms an unconditional basis.
In Section 3, we will give a complete characterization of  unconditional bases formed by power dilation systems for Dirichlet-type spaces.
There are some other function spaces, where
problems on the existence of unconditional bases are raised, such as Hardy spaces $\mathcal{H}^p$ of Dirichlet series and Hardy spaces $H^p(\mathbb{T}^\infty)$ on the infinite  tours, see \cite{AOS,SS}.
We will also prove that a frame formed by a
power dilation system for $H_0^2(\mathbb{D})$ is a Riesz basis, while for $t\neq0$, a power dilation system in $\mathcal{D}_t$ never forms a frame.

\vskip2mm

Here we also mention a related  problem on the Hardy space $H^2(\mathbb{D})$ --- Rudin's orthogonality problem, which is  whether inner functions vanishing at $z=0$ are the only bounded holomorphic functions, up to constant multiple, whose powers are mutually orthogonal in $H^2(\mathbb{D})$. The answer to this problem is yes under some mild conditions \cite{Bou,CKS}, and no in general \cite{Bi1,Bi2,Sun}. The analogous problems on other spaces were also studied. See \cite{GZ}  on the Bergman space, and \cite{CCG} on the Dirichlet space.

\vskip2mm

This paper is arranged as follows.
In Section 2, we list some preliminaries and some preparatory results. In Section 3, we prove Theorem \ref{t1} and give a complete characterization of  unconditional bases and frames formed by power dilation systems for Drichlet-type spaces.
In Section 4, we present an application of our results to an operator moment problem.

\section{Preliminaries}


\subsection{The Hardy space on the infinite torus}
In this subsection, we will introduce the Hardy space $H^2(\mathbb{T}^\infty)$ on the infinite torus, and present some preparatory results.

Let $\mathbb{T}$ denote the unit circle $\{z\in\mathbb{C}:|z|=1\}$, and $\mathbb{T}^\infty$  the product of countably infinitely many circles
$$\mathbb{T}\times\mathbb{T}\times\cdots$$
endowed with the product topology. Obviously $\mathbb{T}^\infty$ is a compact Hausdorff group, and hence it possesses a Haar measure $\rho$. This measure $\rho$ is a product measure. More precisely, let $m$ denote the normalized arc
length measure on $\mathbb{T}$ with $m(\mathbb{T})=1$, and suppose that $E_1, E_2, \cdots, E_n$  are Borel subsets of $\mathbb{T}$, then
$$\rho(E_1\times E_2\times\cdots\times E_n\times\mathbb{T}\times\mathbb{T}\times\cdots )=\prod_{j=1}^{n}m(E_j).$$
Let
$\mathbb{Z}_+^{(\infty)}$ be the set of all finitely supported sequences of non-negative integers, that is, $\mathbb{Z}_+^{(\infty)}=\bigcup_{n=1}^\infty\mathbb{Z}_+^n$.
Define $H^2(\mathbb{T}^\infty)$ to be the closed subspace of $L^2(\mathbb{T}^\infty)=L^2(\mathbb{T}^\infty,\rho)$ spanned by the monomials $\{\zeta^\alpha:\alpha\in\mathbb{Z}_+^{(\infty)}\}$, which constitutes an orthonormal basis for  $H^2(\mathbb{T}^\infty)$.
Two Hardy spaces
 $H^2(\mathbb{T}^\infty)$ and $H^2(\mathbb{D}_2^\infty)$ can be identified with each other via the Poisson
integral.
See \cite{CG,HLS} for more details.

Recall that the Bohr transform $\mathcal{B}$ for the Hardy space
is defined as
\begin{eqnarray*}\mathcal{B}:\quad H_0^2(\mathbb{D})& \rightarrow &
  H^2(\mathbb{T}^\infty),\\
  \sum_{n=1}^\infty a_nz^n& \mapsto & \sum_{n=1}^\infty a_n\zeta^{\alpha(n)}.\end{eqnarray*}
A  function $\eta\in H^2(\mathbb{T}^\infty)$
 is said to be inner if $\eta$ is of modulus $1$ almost everywhere on $\mathbb{T}^\infty$.
The following result can be translated from \cite[Proposition 5.1]{O}.
\begin{prop}\label{p0} Let $f$ be a function in $H_0^2(\mathbb{D})$ with norm $1$. Then  $\{f(z^k)\}_{k\in\mathbb{N}}$ is orthogonal in $H_0^2(\mathbb{D})$ if and only if $\mathcal{B}f$
is
 inner.
\end{prop}

%

\vskip2mm

In what follows we  present some preparatory results, which will be used in the next section.

%
\begin{lem}\label{l1} Suppose $F,G\in L^2(\mathbb{T}^\infty)$. Then $FG\equiv c$ $\text{a.e.}$ for some constant  $c$ if and only if  $$\langle\zeta^{\alpha(i)}F,\zeta^{\alpha(j)}\overline{G}\rangle
_{L^2(\mathbb{T}^\infty)}
=0$$
for any pair  $i$, $j$ of coprime positive integers satisfying $ij>1$. In this case,
$\langle F,\overline{G}\rangle
_{L^2(\mathbb{T}^\infty)}
=c$.
\end{lem}
\begin{proof} It is easy to see that $FG\equiv c\ \text{a.e.}$ if and only if for any pair  $l$, $m$ of positive integers,
\begin{equation}\label{211}\int_{\mathbb{T}^\infty}\zeta^{\alpha(l)}
\overline{\zeta^{\alpha(m)}}F(\zeta)G(\zeta) \mathrm{d}\rho(\zeta)=c\int_{\mathbb{T}^\infty}\zeta^{\alpha(l)}
\overline{\zeta^{\alpha(m)}} \mathrm{d}\rho(\zeta).\end{equation}
Let $d=(l,m)$ be the greatest common divisor of $l$ and $m$, and put $i=\frac ld$, $j=\frac md$. Then $i$ and $j$ are coprime, and for
$\zeta\in\mathbb{T}^\infty$, $$\zeta^{\alpha(l)}
\overline{\zeta^{\alpha(m)}}=\zeta^{\alpha(i)}\zeta^{\alpha(d)}
\overline{\zeta^{\alpha(j)}\zeta^{\alpha(d)}}=\zeta^{\alpha(i)}
\overline{\zeta^{\alpha(j)}}.$$ Hence
 (\ref{211}) is equivalent to
$$\int_{\mathbb{T}^\infty}\zeta^{\alpha(i)}
\overline{\zeta^{\alpha(j)}}F(\zeta)G(\zeta) \mathrm{d}\rho(\zeta)=c\int_{\mathbb{T}^\infty}\zeta^{\alpha(i)}
\overline{\zeta^{\alpha(j)}} \mathrm{d}\rho(\zeta),$$
that is, $$\langle\zeta^{\alpha(i)}F,\zeta^{\alpha(j)}\overline{G}\rangle
_{L^2(\mathbb{T}^\infty)}
=c\langle\zeta^{\alpha(i)},\zeta^{\alpha(j)}\rangle_{L^2(\mathbb{T}^\infty)}.$$
This completes the proof.
\end{proof}

The following result follows from Lemma \ref{l1} and a direct calculation. By virtue of the Bohr transform $\mathcal{B}$, we find that it is actually a restatement of Proposition \ref{p0}.

\begin{prop}\label{c1} Suppose $F=\sum_{n=1}^\infty a_n\zeta^{\alpha(n)}\in H^2(\mathbb{T}^\infty)$. Then
$|F|\equiv c$ $\text{a.e.}$ for some  $c\geq0$ if and only if
and $$\sum_{n=1}^\infty \overline{a_{ni}}a_{nj}=0$$ for any pair  $i$, $j$ of coprime positive
integers. In this case, $\sum_{n=1}^\infty|a_n|^2=c^2$.
\end{prop}

%
%

For each $\tau=(\tau_1, \tau_2, \cdots)\in [0,1]^\infty$, that is, $0\leq \tau_n\leq1$ for $n\in\mathbb{N}$. Write $\tau\zeta=(\tau_1\zeta_1,\tau_2\zeta_2,\cdots)$ and we define an operator $T_{\tau}$ on $H^2(\mathbb{T}^\infty)$ by
setting $$T_{\tau}(\sum_{n=1}^\infty a_n\zeta^{\alpha(n)})=\sum_{n=1}^\infty a_n(\tau\zeta)^{\alpha(n)}.$$ Then $T_{\tau}$ is a positive diagonal
operator with norm less than $1$.
 For a function $F\in H^2(\mathbb{T}^\infty)$, put $F_{\tau}(\zeta)=(T_{\tau}F)(\zeta)=F(\tau\zeta)$.

\begin{prop}\label{c2} Suppose $F\in H^2(\mathbb{T}^\infty)$ and $\tau\in [0,1]^\infty$. Then  $|F_{\tau}|\equiv c\ \text{a.e.}$ for some $c\geq0$ if and only if
$F_{\tau^2}\overline{F}\equiv c^2\ \text{a.e.}$.
\end{prop}
\begin{proof} For any pair  $i$, $j$ of coprime positive integers,
\begin{equation*}
\begin{split}
\tau^{\alpha(ij)}\langle\zeta^{\alpha(i)}F_{\tau},
\zeta^{\alpha(j)}F_{\tau}\rangle_{H^2(\mathbb{T}^\infty)}&=\langle T_{\tau}(\zeta^{\alpha(i)}F),T_{\tau}(\zeta^{\alpha(j)}F)\rangle
_{H^2(\mathbb{T}^\infty)}\\& =\langle T_{\tau}^2(\zeta^{\alpha(i)}F),\zeta^{\alpha(j)}F\rangle
_{H^2(\mathbb{T}^\infty)}\\
&=\tau^{\alpha(i^2)}\langle \zeta^{\alpha(i)}F_{\tau^2},\zeta^{\alpha(j)}F\rangle_{H^2(\mathbb{T}^\infty)}.
\end{split}
\end{equation*}
Thus the proposition follows immediately by Lemma \ref{l1}.
\end{proof}

%

\subsection{Some notions and results from basis theory}
In this subsection, we will introduce some notions from basis theory and list some basic results.

 Suppose that $H$ is a separable Hilbert space.
 A sequence $\{x_k\}_{k\in\mathbb{N}}$ in $H$  is called a Schauder basis for $H$ if to each vector
 $x\in H$ there corresponds a unique sequence of scalars $\{c_k(x)\}_{k\in\mathbb{N}}$ such that
 \begin{equation}\label{basis}
   x=\sum_{k=1}^\infty c_k(x)x_k.
 \end{equation}
 A sequence $\{x_k\}_{k\in\mathbb{N}}$ in $H$  is called a Riesz basis for $H$ if there is an invertible bounded operator
$T\in B(H)$ such that  $\{Tx_k\}_{k\in\mathbb{N}}$ forms an orthonormal basis for $H$.

Let
$H^\infty(\mathbb{T}^\infty)$ denote the set of
bounded  functions in $H^2(\mathbb{T}^\infty)$. Obviously, $H^\infty(\mathbb{T}^\infty)$ with the supreme norm is a Banach algebra.
The result in  \cite{HLS} on Riesz bases formed by dilation system for $L^2(0,1)$  can be translated into the language of the Hardy space as follows.

\begin{thm}[\cite{HLS}]\label{t2} Suppose $f\in H_0^2(\mathbb{D})$. Then  $\{f(z^k)\}_{k\in\mathbb{N}}$ forms an Riesz basis for $H_0^2(\mathbb{D})$ if and only if
$\mathcal{B}f$ is invertible in the algebra
$H^\infty(\mathbb{T}^\infty)$.
\end{thm}

Say a Schauder basis $\{x_k\}_{k\in\mathbb{N}}$ is an unconditional basis for $H$ if the series in  (\ref{basis}) converges
unconditionally for each $x\in H$, that is, for each permutation $\sigma$ of $\mathbb{N}$ the series
$$\sum_{k=1}^\infty c_{\sigma(k)}(x)x_{\sigma(k)}$$ converges to $x$.

Finally, a sequence $\{x_k\}_{k\in\mathbb{N}}$ in  $H$  is called a frame for $H$ if there exist constants $A,B>0$ such that
\begin{equation}\label{frame definition}
  A\|x\|^2\leq\sum_{k=1}^{\infty}|\langle x, x_k\rangle|^2\leq B\|x\|^2,\quad x\in H.
\end{equation}
It is clear that every frame is complete, and every Riesz basis is a frame. The special case when $A=B=1$ in (\ref{frame definition}) defines the Parseval frame. That is to say, a Parseval frame for $H$ is a sequence $\{x_k\}_{k\in\mathbb{N}}$ in  $H$ satisfying
$$\|x\|^2=\sum_{k=1}^{\infty}|\langle x, x_k\rangle|^2,\quad x\in H.$$
The following two results will be needed in the sequel, see \cite[Lemma 3.6.9, Theorem 7.1.1]{Ch}.

\begin{lem}\label{uncon} A  Schauder basis $\{x_k\}_{n\in\mathbb{N}}$ for  $H$ is a Riesz basis if and only if
it is an unconditional basis bounded from both above and below, that is, $$0<\inf_{k\in\mathbb{N}}\|x_k\|\leq
\sup_{k\in\mathbb{N}}\|x_k\|<\infty.$$
\end{lem}

\begin{lem} \label{frame Riesz}
  Let $\{x_k\}_{n\in\mathbb{N}}$ be a frame for $H$. Then the following are equivalent:
  \begin{itemize}
    \item [(1)] $\{x_k\}_{n\in\mathbb{N}}$ is a Riesz basis for $H$;
    \item [(2)] $\{x_k\}_{n\in\mathbb{N}}$ biorthogonal system, that is, there exists a sequence $\{y_k\}_{n\in\mathbb{N}}$ in $H$, such that
        $$\langle x_k,y_l\rangle=\delta_{kl},\quad k,l=1,2,\cdots;$$
    \item [(3)] $\{x_k\}_{n\in\mathbb{N}}$ is $\omega$-independent, that is, whenever
 $\sum_{k=1}^\infty c_kx_k$ converges to $0$ in $H$-norm for some sequence $\{c_k\}_{k\in\mathbb{N}}$ of scalars, then
 necessarily $c_k=0$ for all $k\in\mathbb{N}$.
  \end{itemize}
\end{lem}

\section{Main results}
In this section we will prove Theorem \ref{t1} restated below, and
give a complete characterization of  unconditional bases and frames formed by power dilation systems for Drichlet-type spaces.
%
\vskip2mm
\noindent\textbf{Theorem \ref{t1}.} \emph{Suppose $t\neq0$ and $f\in \mathcal{D}_t$. Then  $\{f(z^k)\}_{k\in\mathbb{N}}$ is orthogonal in $\mathcal{D}_t$ if and only if $f(z)=cz^N$ for some constant $c$ and some positive integer $N$.}
\begin{proof}
 Let $$f(z)=\sum_{n=1}^\infty a_nz^n$$ be the Taylor expansion of $f$. With no loss of generality, assume that $f$ is not identically zero. Taking an arbitrary
pair  $i$, $j$ of coprime positive
integers satisfying $ij>1$, by the hypotheses we have
$$0=\langle f(z^{ki}),f(z^{kj})\rangle=\sum_{n=1}^\infty \overline{a_{ni}}a_{nj}(nkij+1)^t.$$
By rearranging the above equality, we see that
\begin{equation}\label{241}\sum_{n=1}^\infty \overline{a_{ni}}a_{nj}(nij+\frac 1k)^t=0.\end{equation}
 The dominated convergence theorem implies that the left side of (\ref{241}) tends to $$\sum_{n=1}^\infty \overline{a_{ni}}a_{nj}(nij)^t$$ as $k\rightarrow \infty$,
and hence
\begin{equation}\label{242}\sum_{n=1}^\infty \overline{a_{ni}}a_{nj}(nij)^t=0.\end{equation}
 Now by subtracting (\ref{242}) from (\ref{241}), and then multiplying  both sides  by $k$, we have
\begin{equation}\label{243}\sum_{n=1}^\infty \overline{a_{ni}}a_{nj}[(nij+\frac 1k)^t-(nij)^t]k=0.\end{equation}
Since $$[(nij+\frac 1k)^t-(nij)^t]k\rightarrow t(nij)^{t-1}\quad(k\rightarrow\infty),$$ again by the dominated convergence theorem, the left side of (\ref{243}) tends to
$$t\cdot\sum_{n=1}^\infty \overline{a_{ni}}a_{nj}(nij)^{t-1}$$ as $k\rightarrow \infty$.  Thus, we have
\begin{equation}\label{244}\sum_{n=1}^\infty \overline{a_{ni}}a_{nj}(nij)^{t-1}=0.\end{equation}

Put $$F=\sum_{n=1}^\infty a_nn^{\frac t2}\zeta^{\alpha(n)}\in H^2(\mathbb{T}^\infty).$$ By Proposition \ref{c1} and (\ref{242}), we have $$|F|^2\equiv\sum_{n=1}^\infty|a_n|^2n^t\quad \text{a.e.}.$$
Similarly, by Proposition \ref{c1} and (\ref{244}),
$$|F_{\tau}|^2\equiv\sum_{n=1}^\infty|a_n|^2n^{t-1}\quad  \text{a.e.},$$
where
$\tau=(\frac{1}{\sqrt{p_1}},\frac{1}{\sqrt{p_2}},\cdots)$ and $p_j\ (j\in\mathbb{N})$ is the $j$-th prime.
Then Corollary \ref{c2} shows that $$F_{\tau^2}\overline{F}\equiv \sum_{n=1}^\infty|a_n|^2n^{t-1}\quad \text{a.e.}.$$
Therefore, $$T_{\tau^2}F=F_{\tau^2}=\frac{F_{\tau^2}\overline{F}}{|F|^2}\cdot F= \frac{\sum_{n=1}^\infty|a_n|^2n^{t-1}}{\sum_{n=1}^\infty|a_n|^2n^t}F,$$
which implies that $F$ is a eigenvector of the diagonal operator $T_{\tau^2}$. It follows that
$F$ is a monomial, and hence $f(z)=a_Nz^N$ for some positive integer $N$.
The proof is complete.
\end{proof}

\vskip2mm

Recall that when $t\neq0$, a power dilation system in $\mathcal{D}_t$ is not
bounded from both above and below, and thus
cannot form a Riesz basis.
In what follows, we will give a complete characterization of  unconditional bases formed by power dilation systems for $\mathcal{D}_t$.

To continue, we need a bit more notations. For $t\in\mathbb{R}$ we
define
\begin{eqnarray*}S_t:\quad \mathcal{D}_t& \rightarrow &
  H^2(\mathbb{D}),\\
  \sum_{n=0}^\infty a_nz^n& \mapsto & \sum_{n=0}^\infty {a_n}n^{\frac t2}z^n.\end{eqnarray*} Then
   each $S_t$ is an invertible bounded operator.
Putting $\mathcal{B}_t=\mathcal{B}S_t\ (t\in\mathbb{R})$, we have the following theorem.

\begin{thm}\label{t0} Suppose $t\in\mathbb{R}$ and $f\in \mathcal{D}_t$. Then  $\{f(z^k)\}_{k\in\mathbb{N}}$ forms an unconditional basis for $\mathcal{D}_t$ if and only if  $\mathcal{B}_t f$ is invertible in
$H^\infty(\mathbb{T}^\infty)$.
\end{thm}
\begin{proof} For each $k\in\mathbb{N}$, let  $C_k$ be the $k$-power dilation operator on $\mathrm{Hol}(\mathbb{D})$:
$$C_kf(z)=f(z^k),\quad f\in\mathrm{Hol}(\mathbb{D}).$$ A direct computation gives
\begin{equation}\label{e1}
                      S_tC_k=k^{\frac t2}C_kS_t,\quad k=1, 2, \cdots
\end{equation} on $\mathcal{D}_t$.
In what follows, we will show the following four conditions
are equivalent, which would complete the proof.
\begin{itemize}
  \item [(1)] $\{f(z^k)\}_{k\in\mathbb{N}}$ forms an unconditional basis for $\mathcal{D}_t$;
  \item [(2)] there exists some sequence $\{\lambda_k\}_{k\in\mathbb{N}}$ of complex numbers  (which may depend on $t$), such that $\{\lambda_kS_tC_kf\}_{k\in\mathbb{N}}$ forms a  Riesz basis for $H_0^2(\mathbb{D})$;
  \item [(3)] $\{k^{-\frac t2}S_tC_kf\}_{k\in\mathbb{N}}$ forms a  Riesz basis for $H_0^2(\mathbb{D})$;
  \item [(4)] $\mathcal{B}_t f$  is invertible in
$H^\infty(\mathbb{T}^\infty)$.
\end{itemize}

\noindent(1)$\Leftrightarrow$(2): By the invertibility of $S_t$, (2) is equivalent to the following:
\begin{itemize}
  \item [(2')] there exists some sequence of complex number $\{\lambda_k\}_{k\in\mathbb{N}}$, such that $\{\lambda_kf(z^k)\}_{k\in\mathbb{N}}$ forms a Riesz
basis for $\mathcal{D}_t$.
\end{itemize}
The equivalence of (1) and (2') is a direct consequence of Lemma \ref{uncon}, and then implies the equivalence of (1) and (2).
\vskip2mm

\noindent(3)$\Leftrightarrow$(4): By (\ref{e1}),
$$k^{-\frac t2}(S_tC_kf)(z)=(C_kS_tf)(z)=(S_tf)(z^k)$$
for each $k\in\mathbb{N}$.
The equivalence of (3) and (4) follows immediately from Theorem \ref{t2}.
\vskip2mm

\noindent(3)$\Rightarrow$(2): Obvious.
\vskip2mm

\noindent(2)$\Rightarrow$(3): Assume that $\{\lambda_kS_tC_kf\}_{k\in\mathbb{N}}$ forms a  Riesz basis for some sequence $\{\lambda_k\}_{k\in\mathbb{N}}$ of complex numbers. Again by (\ref{e1}), we have \begin{equation*}
\begin{split}
   \|\lambda_kS_tC_kf\|_{H^2(\mathbb{D})} &=|\lambda_kk^{\frac t2}|\cdot\|C_k S_tf\|_{H^2(\mathbb{D})} \\
     &=|\lambda_kk^{\frac t2}|\cdot\|S_tf\|_{H^2(\mathbb{D})},\quad k=1, 2, \cdots,
\end{split}
\end{equation*} since  each $C_k$ is an isometry on $H^2(\mathbb{D})$. This implies that
 $$0<\inf_{k\in\mathbb{N}}|\lambda_kk^{\frac t2}|\leq
\sup_{k\in\mathbb{N}}|\lambda_kk^{\frac t2}|<\infty,$$
and hence  $$0<\inf_{k\in\mathbb{N}}|\lambda_k^{-1}k^{-\frac t2}|\leq
\sup_{k\in\mathbb{N}}|\lambda_k^{-1}k^{-\frac t2}|<\infty.$$
By the equality $k^{-\frac t2}S_tC_kf=(\lambda_k^{-1}k^{-\frac t2})\,(\lambda_k S_tC_kf)$ and  the above inequality, we see  that the new sequence $\{k^{-\frac t2}S_tC_kf\}_{k\in\mathbb{N}}$  forms a Riesz basis for $H_0^2(\mathbb{D})$.
\end{proof}

Now we  consider the problem when does a power dilation system form a frame.
It was mentioned in \cite[Section 5]{HLS} that every complete power dilation system in $H_0^2(\mathbb{D})$ has a biorthogonal system.
Therefore Lemma \ref{frame Riesz} implies that
every frame formed by a power dilation system for $H_0^2(\mathbb{D})$ is a Riesz basis. This can also be deduced by combining Lemma \ref{frame Riesz} with the following result.
 \begin{lem}\label{independent}
   Suppose that $t\in\mathbb{R}$ and $f(z)=\sum_{n=1}^\infty a_nz^n\in\mathcal{D}_t$. If $a_1\neq0$ we have that $\{f(z^k)\}_{k\in\mathbb{N}}$ is $\omega$-independent in $\mathcal{D}_t$.
 \end{lem}
 \begin{proof}
   Assume that $a_1\neq0$ and
   $\sum_{k=1}^\infty c_kf(z^k)$ converges to the zero function in $\mathcal{D}_t$-norm for some sequence $\{c_k\}_{k\in\mathbb{N}}$ of scalars. Then for each $n\in\mathbb{N}$,
   $$0=\lim_{N\rightarrow\infty}\langle \sum_{k=1}^{N} c_kf(z^k), z^n\rangle_{\mathcal{D}_t}
   =\lim_{N\rightarrow\infty}\sum_{k=1}^{N}c_k\langle f(z^k), z^n\rangle_{\mathcal{D}_t}
   =\|z^n\|_{\mathcal{D}_t}^2\cdot\sum_{k|n}c_k\,a_{\frac{n}{k}}.$$
   Considering $n=1$ in the above identity, one easily see that $c_1=0$. By the induction, we come to the desired conclusion.
 \end{proof}

More generally, we obtain a complete characterization of frames formed by power dilation systems for Dirichlet-type spaces, which is summarized as follows.
\begin{thm} \label{t3}
  Suppose $t\in\mathbb{R}$ and $f\in\mathcal{D}_t$. Then
  \begin{itemize}
    \item [(1)] if $t=0$, $\{f(z^k)\}_{k\in\mathbb{N}}$ is a frame for $H_0^2(\mathbb{D})$ if and only if $\{f(z^k)\}_{k\in\mathbb{N}}$ is a Riesz basis for $H_0^2(\mathbb{D})$, if and only if
$\mathcal{B}f$ is invertible in the algebra
$H^\infty(\mathbb{T}^\infty)$.
    \item [(2)] if $t\neq0$, $\{f(z^k)\}_{k\in\mathbb{N}}$ never forms a frame for $\mathcal{D}_t$.
  \end{itemize}
\end{thm}

Indeed, Theorem \ref{t3} (2) can be proved directly by using the definition of frames. For this, assume conversely
there  exist $f(z)=\sum_{n=1}^\infty a_nz^n\in\mathcal{D}_t\ (t\neq0)$ and  constants $A,B>0$, such that \begin{equation}\label{frame}
  A\|h\|_{\mathcal{D}_t}^2\leq\sum_{k=1}^{\infty}|\langle h, f(z^k)\rangle_{\mathcal{D}_t}|^2\leq B\|h\|_{\mathcal{D}_t}^2,\quad h\in \mathcal{D}_t.
\end{equation}
It is clear that $\{f(z^k)\}_{k\in\mathbb{N}}$ converges weakly to $0$ in $\mathcal{D}_t$, and then is  uniformly norm-bounded. This implies that $t$ is necessarily less than $0$. By choosing a prime $p>(\frac{A}{2B})^{\frac{1}{t}}$ and substituting $h=z^{p^n}\ (n\in\mathbb{N})$ in (\ref{frame}), we have
$$A(p^n+1)^{-t}\leq\sum_{i=1}^{n}|a_{p^i}|^2\leq B(p^n+1)^{-t},\quad n=1,2,\cdots.$$
Then there exists $N\in\mathbb{N}$ such that for each $n\geq N$,
$$Ap^{-nt}\leq\sum_{i=1}^{n}|a_{p^i}|^2\leq 2Bp^{-nt},$$
which yields that
$$|a_{p^n}|^2\geq Ap^{-nt}-2Bp^{-(n-1)t},\quad n=N+1,N+2,\cdots.$$
Therefore, $$\|f\|_{\mathcal{D}_t}^2
\geq \sum_{n=N+1}^{\infty}|a_{p^n}|^2(p^n+1)^t
\geq2^t \sum_{n=N+1}^{\infty}|a_{p^n}|^2p^{nt}\geq2^t \sum_{n=N+1}^{\infty}(A-2Bp^t)=\infty,$$ which is a contradiction.
\vskip2mm

Since a Parseval frame  is an orthonormal  basis  if and only if it is also a Riesz basis, Theorem \ref{t3} (1) gives the following.

\begin{cor}
  Suppose  $f\in H_0^2(\mathbb{D})$. Then $\{f(z^k)\}_{k\in\mathbb{N}}$ is a Parseval frame for $H_0^2(\mathbb{D})$ if and only if $\{f(z^k)\}_{k\in\mathbb{N}}$ is an orthonormal basis for $H_0^2(\mathbb{D})$, if and only if
$\mathcal{B}f$ is inner in $H^2(\mathbb{T}^\infty)$.
\end{cor}

To end this section, we consider complete  power dilation systems in  Dirichlet-type spaces.
 By (\ref{e1}), for each $f\in \mathcal{D}_t$ we have $$S_t\big(\mathrm{span}_{k\in\mathbb{N}}\{f(z^k)\}\big)
=\mathrm{span}_{k\in\mathbb{N}}\{(S_tf)(z^k)\},$$  which implies that $\{f(z^k)\}_{k\in\mathbb{N}}$ is complete in $\mathcal{D}_t$ if and only if $\{(S_tf)(z^k)\}_{k\in\mathbb{N}}$ is complete in $H_0^2(\mathbb{D})$. Therefore the completeness problems of power dilation systems on  Dirichlet-type spaces are all equivalent, and thus all equivalent to the Wintner-Beurling problem.
This problem is still not completely solved, see \cite{DG1,DG2} for recent progresses.

\section{Application to an operator moment problem}
Given $t\in \mathbb{R}$, $f\in \mathcal{D}_t$ with $\|f\|_{\mathcal{D}_t}=1$ and a sequence $\{\lambda_k\}_{k\in\mathbb{N}}$ of complex numbers, we consider the following  operator moment problem on $\mathcal{D}_t$:
\begin{equation}\label{moment}
  Tz^k=\lambda_kf(z^k),\quad k=1,2,\cdots
\end{equation}
That is to say, we seek conditions to make  equations in (\ref{moment}) have a common solution $T$ that is a bounded linear operator on $\mathcal{D}_t$, maybe with additional properties.
Note that by Lemma \ref{independent}, if the equations (\ref{moment}) has a  solution $T$, then $T$ is necessarily injective.

\begin{thm}
  \begin{itemize}
    \item [(1)] If $t=0$ then the equations (\ref{moment}) has an solution $T$ isometric on $H_0^2(\mathbb{D})$ if and only if $\mathcal{B}f$ is inner in $H^2(\mathbb{T}^\infty)$ and $|\lambda_k|=1$ for all $k\in\mathbb{N}$.
     \item [(2)] If $t\neq0$ then the equations (\ref{moment}) has an  solution $T$ isometric on $\mathcal{D}_t$ if and only if $f=cz^N$ for some positive integer $N$ and some constant $c$ with
         $|c|=(N+1)^{-\frac{t}{2}}$, and $|\lambda_k|=(\frac{kN+k+N+1}{kN+1})^\frac{t}{2}$ for each $k\in\mathbb{N}$.
    \item [(3)] The equations (\ref{moment}) has a solution $T$ which is a surjective or invertible operator if and only if $\mathcal{B}_t f$ is invertible in
$H^\infty(\mathbb{T}^\infty)$, and $\{\lambda_k\}_{k\in\mathbb{N}}$ is bounded from both above and below.
    \item [(4)] If $\{\lambda_k\}_{k\in\mathbb{N}}$ is bounded from both above and below,
         then the equations (\ref{moment}) has a  solution $T$ if and only if
 $\mathcal{B}_t f$ is bounded.
  \end{itemize}
\end{thm}
\begin{proof} (1) and (2) immediately follow from Proposition \ref{p0} and Theorem \ref{t1} since the existence of a solution $T$ isometric on $\mathcal{D}_t$ would imply that the power dilation system $\{f(z^k)\}_{k\in\mathbb{N}}$ is orthogonal in $\mathcal{D}_t$.
\vskip2mm

(3)  By the definition, it is clear that the equations (\ref{moment}) have  an solution $T$ which is an invertible operator if and only if
$\{\frac{\lambda_k}{\|z^k\|_{\mathcal{D}_t}}f(z^k)\}_{k\in\mathbb{N}}$ forms a Riesz basis for $\mathcal{D}_t$. Form the proof of Theorem \ref{t0}, we see that this is further equivalent to that $\mathcal{B}_t f$ is invertible in
$H^\infty(\mathbb{T}^\infty)$, and
$$0<\inf_{k\in\mathbb{N}}\frac{|\lambda_k k^{\frac t2}|}{\|z^k\|_{\mathcal{D}_t}}\leq
\sup_{k\in\mathbb{N}}\frac{|\lambda_k k^{\frac t2}|}{\|z^k\|_{\mathcal{D}_t}}<\infty,$$
which proves (3).
  \vskip2mm
  (4) Assume that $\{\lambda_k\}_{k\in\mathbb{N}}$ is bounded from both above and below. Then the equations (\ref{moment}) has a  solution $T$ if and only if
   $$\widetilde{T}z^k=f(z^k),\quad k=1,2,\cdots$$
  has a  solution $\widetilde{T}$, and hence is equivalent to that
  $$S\mathcal{B}_tz^k=\mathcal{B}_tC_kf,\quad k=1,2,\cdots$$
  has a  solution $S$ (defined on the Hardy space $H^2(\mathbb{T}^\infty)$). A direct calculation gives that such an operator $S$ exactly satisfies $Sp=p\,\mathcal{B}_tf$
  for any polynomial $p$.
  This implies that the existence of the solution $S$ is equivalent to that $\mathcal{B}_tf$ is a multiplier on $H^2(\mathbb{T}^\infty)$. The proof is complete due to the fact that $H^\infty(\mathbb{T}^\infty)$ coincides with the multiplier algebra of $H^2(\mathbb{T}^\infty)$ \cite{Ni}.
\end{proof}

In particular, we have following characterizations for the cases of the Hardy space and the Bergman space.

In the case $\mathcal{D}_0=H_0^2(\mathbb{D})$ and
$\lambda_k=1$ for all $k\in\mathbb{N}$, the equations (\ref{moment}) has
\begin{itemize}
  \item [(1)] a  solution $T$ if and only if
 $\mathcal{B} f$ is bounded;
  \item [(2)] an solution $T$ isometric on $H_0^2(\mathbb{D})$ if and only if $\mathcal{B}f$ is inner in $H^2(\mathbb{T}^\infty)$;
  \item [(3)] a solution $T$ which is a surjective or invertible operator if and only if $\mathcal{B} f$ is invertible in
$H^\infty(\mathbb{T}^\infty)$.
\end{itemize}

In the case $\mathcal{D}_{-1}=L_{a,0}^2(\mathbb{D})$ and
$\lambda_k=1$ for all $k\in\mathbb{N}$, the equations (\ref{moment}) does not have an solution $T$ isometric on $L_{a,0}^2(\mathbb{D})$, and has \begin{itemize}
  \item [(1)] a  solution $T$ if and only if
 $\mathcal{B}_{-1} f$ is bounded;
  \item [(2)] a  solution $T$ which is a surjective or invertible operator if and only if $\mathcal{B}_{-1} f$ is invertible in
$H^\infty(\mathbb{T}^\infty)$.
\end{itemize}
\vskip2mm

\noindent \textbf{Acknowledgement} This work is  partially supported by
 Natural Science Foundation of China.

\vskip3mm \noindent{Hui Dan, College of Mathematics, Sichuan
University, Chengdu, Sichuan, 610065, China,
   E-mail:  danhuimath@gmail.com

\noindent Kunyu Guo, School of Mathematical Sciences, Fudan
University, Shanghai, 200433, China, E-mail: kyguo@fudan.edu.cn

\end{document}